\documentclass[10pt,a4paper]{amsart}
\usepackage{fullpage} 
\usepackage{amsmath}
\usepackage{amsthm}
\usepackage{amssymb}
\usepackage{mathtools}
\usepackage[only,mapsfrom]{stmaryrd}
\usepackage{graphicx}
\usepackage{cancel} 
\usepackage{color} 
\usepackage[all]{xy} 
\usepackage{tikz}
\usetikzlibrary{arrows,calc}
\usetikzlibrary{cd}
\usepackage{nicefrac}
\usepackage{enumitem}
\usepackage[colorlinks,citecolor=blue,linkcolor=blue,urlcolor=blue,filecolor=blue,breaklinks]{hyperref}
\usepackage{caption}
\usepackage{comment}
\usepackage{tensor}

\numberwithin{equation}{section}
\numberwithin{figure}{section}

\newtheorem{thm}{Theorem}[section]
\newtheorem{athm}{Theorem}

\newtheorem{lem}[thm]{Lemma}
\newtheorem{prop}[thm]{Proposition}

\newtheorem{prob}[thm]{Problem}
\newtheorem{obs}[thm]{Observation}
\newtheorem*{thm*}{Theorem}
\newtheorem*{conj*}{Conjecture}
\newtheorem*{cor*}{Corollary}
\newtheorem*{ques*}{Question}
\newtheorem*{claim*}{Claim}

\theoremstyle{definition}
\newtheorem{rem}[thm]{Remark}
\newtheorem{defn}[thm]{Definition}

\newtheorem{ex}[thm]{Example}

\newtheorem*{rem*}{Remark}

\newcommand{\cB}{\mathcal{B}}

\newcommand{\cE}{\mathcal{E}}

\newcommand{\cG}{\mathcal{G}}
\newcommand{\cH}{\mathcal{H}}

\newcommand{\cK}{\mathcal{K}}

\newcommand{\cO}{\mathcal{O}}

\newcommand{\bN}{\mathbb{N}}

\newcommand{\bR}{\mathbb{R}}

\newcommand{\bZ}{\mathbb{Z}}

\newcommand{\incl}[3][right]%
{%
\draw[<-,>=#1 hook] #2 to ($ #2!0.5!#3 $);
\draw[->] ($ #2!0.5!#3 $) to #3;%
}
\newcommand{\inclusion}[5][right]%
{%
\draw[<-,>=#1 hook] #4 to ($ #4!0.5!#5 $) node[#2,font=\small]{#3};
\draw[->,>=stealth'] ($ #4!0.5!#5 $) to #5;%
}

\renewcommand{\geq}{\geqslant}
\renewcommand{\leq}{\leqslant}

\title{Ample groupoids that are neither almost finite nor purely infinite}

\author{Xiaolei Wu}
\address{Shanghai Center for Mathematical Sciences, Jiangwan Campus, Fudan University, No.2005 Songhu Road, Shanghai, 200438, P.R. China}
\email{xiaoleiwu@fudan.edu.cn}

\author{Mengfei Zhao}
\address{Shanghai Center for Mathematical Sciences, Jiangwan Campus, Fudan University, No.2005 Songhu Road, Shanghai, 200438, P.R. China}
\email{mfzhao22@m.fudan.edu.cn}

\author{Xin Ma}
\address{Institute for Advanced Study in Mathematics, Harbin Institute of Technology, Harbin, 150001, P.R. China}
\email{xma17@hit.edu.cn}

\date{March 2026}

\begin{document}

\begin{abstract}
We investigate a question posed by Matui concerning the existence of minimal ample groupoids that are neither almost finite nor purely infinite, along with several natural variations of this problem. We begin by noting that there exists minimal, effective, ample transformation groupoids with this property. Moreover, such examples can be chosen to be either principal or amenable. We then construct new examples of essentially principal ample groupoids that are likewise neither almost finite nor purely infinite. These examples rely on the recent twisted topological groupoid construction of Palmer and Wu. In particular, our examples do not arise from transformation groupoids.
\end{abstract}
\maketitle

\section*{Introduction}

The notion of topological groupoids goes back to Ehresmann \cite{Ehresmann59}. It has  applications in many areas of mathematics, such as differential geometry \cite{Mackenzie87}, geometric group theory \cite{Haefliger91} and $C^\ast$-algebra \cite{Renault80}. More recently, Matui has generalized the notion of topological full group of Giordano--Putnam--Skau \cite{GiordanoPutnamSkau99} to the setting of effective \'etale groupoid \cite{Matui2012}. Since then, much work has been done in studying such groupoids and their topological full groups. A highlight of this line of research is \cite{JuschenkoMonod13} in which Juschenko and Monod found the first examples of finitely generated simple amenable groups.  

Among all classes of minimal ample groupoids, two families  have been introduced and extensively studied: \emph{almost finite groupoids}~\cite{Matui2012,Matui2015b,Matui2016} and \emph{purely infinite groupoids}~\cite{Matui2015b,Matui2016}. Note that these two notions have been extended to non-ample groupoids, respectively, in \cite{MW20} and \cite{Ma22}.
Both classes exhibit deep relationships with group theory, $C^*$-algebras, and dynamical systems; see~\cite{Matui15} and the references therein for further background. 
In particular, these notions play a central role in the structure and classification theory of groupoid $C^*$-algebras, as demonstrated in~\cite{MW20,Ma22}. 
Furthermore, it has been shown that for a broad class of unital separable simple nuclear $C^*$-algebras satisfying the UCT, including all strongly self-absorbing $C^*$-algebras satisfying the UCT, their classifiability by the Elliott invariant can be characterized by the almost finiteness or the pure infiniteness of their certain associated groupoid models in \cite{MW24} via a uniform notion, called \textit{almost elementariness} introduced in \cite{MW20}.

On the other hand, there is a fundamental difference between these groupoids: an almost finite groupoid admits approximations of every compact subset by compact elementary groupoids, whereas a purely infinite groupoid possesses an abundance of compact open bisections. This contrast is also reflected in their topological full groups, such as finiteness properties \cite{Matui2006,Matui2015b,Matui15} and amenability \cite{Matui2015b,JuschenkoMonod13}.  Despite these differences, the two classes share several important features. For example, the commutator subgroup of their topological full groups are  simple \cite{Matui15}; both admit comparison and satisfy the AH-conjecture \cite{Li2024}. Moreover, it seems that most known examples of effective minimal ample groupoids belong to one of these two classes. This leads Matui to ask the following question:

\begin{prob}\label{matui-prob} 
    Does there exist an effective minimal (amenable) ample groupoid which is neither almost finite nor purely infinite?
\end{prob}

The question appeared on Matui's lecture notes \cite[Problem 3.6]{Matui25} (without the amenability condition) and \cite[Problem 3.8]{Matui25sh}. As we will see in \S \ref{sec:tranf}, there are many such examples already coming from transformation groupoids. Such examples of transformation groupoids can be even chosen to be  minimal, principal (but not amenable) and satisfies dynamical comparison. See Example \ref{ex:odometer} below.  These minimal principal examples are known as odometers in the community of dynamical systems. But to the best knowledge of authors, they have not been used as counterexamples to Problem \ref{matui-prob}.

In addition, Matui's original definition of \emph{essential principality} for groupoids, introduced for instance in~\cite{Renault80} and~\cite{Matui15}, coincides with the notion of \emph{effectivity} for groupoids. 
In this work, however, we adopt a definition of essential principality that is more closely aligned with the concept of \emph{essential freeness} in ergodic theory (see Definition~\ref{defn:free} below). 
This choice is motivated by the fact that essential freeness is known to be a necessary condition for almost finiteness (see, e.g.,~\cite[Lemma~2.2]{Joseph24}). 
In view of this, minimal topologically free but non‑essentially free actions of amenable groups yield amenable counterexamples to Problem~\ref{matui-prob} (see~\cite{Joseph24,HirshbergWu24}). 
It is therefore natural and more motivated to study Matui’s question in the framework of essential principal groupoids. 

On the other hand, it is known that there are obstructions for important $C^*$-algebras, e.g., Jiang-Su algebra $\mathcal{Z}$, to be written as transformation groupoid $C^*$-algebras (see, e.g., \cite[Theorem C]{MW24}). However, it follows from \cite[Theorem A]{MW24} that they can be always written  as almost elementary groupoid $C^*$-algebras. Recall that the notion of almost elementariness is equivalent to either the almost finiteness or the pure infiniteness in the ample case, depending on the existence of invariant probability measures on the unit space (see \cite{MW20,Ma22,MW24}). Hence, from $C^*$-algebraic perspective, it is worthwhile to further explore Matui’s question in the context of non‑transformation groupoids.

Taking these aspects into account, we restate Matui’s question in the following form:

\begin{prob}\label{main-prob} 
    Does there exist an essentially principal minimal (amenable) ample groupoid which is not a transformation groupoid and is neither almost finite nor purely infinite?
\end{prob}

We will recall the corresponding notions in \S 1. If it turned out that there were no examples of such groupoids, there would be a dichotomy among all essentially principal minimal (amenable) ample groupoids that are not transformation groupoids.   Our main theorem provides many such groupoids that do not arise from the transformation groupoid construction.



\begin{athm}\label{thm:maina}
    There exists effective minimal ample  groupoids that do not arise from tranformation groupoids and are neither almost finite nor purely infinite. Some can be even made to be essentially principal.
\end{athm}

Our example uses the recent twisted topological groupoid construction of Palmer and Wu \cite[Remark 2.15]{PalmerWu25}, which in turn was inspired by Belk and Zaremsky's construction of twisted Brin--Thompson groups \cite{BelkZaremsky22}. We establish the fundamental properties of Palmer--Wu's construction in \S \ref{section:basi-prop}, which may be of independent interest. A key new feature of our examples is that they are not transformation groupoids by definition. Indeed, they have the form of $\cG\rtimes \Gamma$, in which $\cG$ is not a trivial groupoid, i.e.,  $\cG\neq \cG^{(0)}$. To be more specific,  We prove Theorem \ref{thm:maina} by establishing the following.

\begin{athm}\label{main:thmb}
       Let $\Gamma$ be a countable non-amenable group, $\cG$ a minimal principal almost finite effective groupoid.  Then the twisted topological groupoid $\Gamma\cG\rtimes \Gamma$ is a minimal effective ample groupoid that is neither almost finite nor purely infinite, where $\Gamma\cG$ is not a trivial groupoid in the sense that $\Gamma\cG\neq(\Gamma\cG)^{(0)}$. Furthermore, if $\cG$ is essentially principal and $\cG$ has unique $\cG$-invariant probability measure, then $\Gamma\cG\rtimes \Gamma$ is also essentially principal.
\end{athm}

\begin{rem}\label{rem:tw-trf-iso}
  See Examples \ref{ex:non a.f. and non p.i.}, \ref{ex: TG non transformation} and  for some concrete examples of such groupoids.  Although our examples are not transformation groupoids by definition, they might still be secretly  isomorphic to some transformation groupoids; see Remark \ref{rmk: final remark} for a discussion.
\end{rem}

Further properties of twisted topological groupoids and their  full groups will be studied in \cite{WuZhao25+}.

\subsection*{Acknowledgments.}

Wu is currently a member of LMNS and is supported by NSFC No.12326601. He thanks Martin Palmer for discussions on twisted topological groupoids. Ma is supported by NSFC No. 12571133.  We thank Hiroki Matui for a wonderful lecture series at Fudan University from which we learned Problem \ref{matui-prob}  (without the amenability
condition). We also thank Jianchao Wu for some stimulating discussions. We further thank him and Kang Li for pointing out Joseph's example \cite{Joseph24}. Finally, we would like to thank the anonymous referee for the very helpful comments and insightful questions regarding whether a twisted groupoid in this paper could be isomorphic to a transformation groupoid. These remarks inspired  Remark~\ref{rmk: final remark}.

\section{ topological groupoids}

In this section, we recall the basics of topological groupoids. See \cite{Matui15,Renault80} for more information. 

A groupoid is a small category whose morphisms are all invertible. As usual, we identify the groupoid with its set of morphisms, denoted by $\cG$, and view its set of objects (also called units) $\cG^{(0)}$ as a subset of $\cG$ by identifying objects with the corresponding identity morphisms. By definition, a groupoid $\cG$ comes equipped with range and source maps $r \colon \cG \to \cG^{(0)}$, $s \colon \cG \to \cG^{(0)}$, a multiplication map
\[
\cG \tensor[_s]{\times}{_r} \cG = \{ (g_1,g_2)\mid s(g_1) = r(g_2)\} \longrightarrow \cG, \qquad (g_1,g_2) \longmapsto g_1g_2
\]
and an inversion map $\cG \to \cG \colon g \mapsto g^{-1}  $ satisfying $r(g^{-1}) = s(g), s(g^{-1}) =r(g), gg^{-1} = r(g)$ and $g^{-1} g = s(g)$. These structure maps must satisfy the usual list of axioms so that $\cG$ is a small category; see \cite[\S 1.1]{Renault80} for more details.

\begin{defn}
A \emph{topological groupoid} is a groupoid equipped with a topology making the composition and inversion maps continuous. In addition, throughout the paper, we will always assume that the topology on $\cG$ makes $\cG$ and its subspace $\cG^{(0)}$ (which we call the \emph{unit space}) locally compact and Hausdorff.
\end{defn}

\begin{defn}
A topological groupoid $\cG$ is called \emph{étale} if it is a second countable locally compact Hausdorff groupoid such that the source and range maps are local homeomorphisms.
\end{defn}

We note that the definition implies that the unit space $\cG^{(0)}$ of an étale groupoid is an open subspace of the whole space $\cG$.

\begin{defn}
An open subspace $U \subseteq \cG$ is called an \emph{open bisection} if the restricted range and source maps $r|_U \colon U\to r(U)$ and $s|_U \colon U\to s(U)$ are homeomorphisms.
\end{defn}

If the groupoid $\cG$ is \'etale, then $\cG$ has a basis for its topology consisting of open bisections. We note that open bisections are always locally compact and Hausdorff because they are homeomorphic to open subspaces of the unit space. For $x\in \cG^{(0)}$, the set $r(\cG x) := \{r(gx) \mid g\in \cG\}$ is called the \emph{$\cG$-orbit} of $x$. 

\begin{defn}
Let $\cG$ be a topological groupoid, it is \emph{minimal} if  every $\cG$-orbit is dense in $\cG^{(0)}$.    
\end{defn}

\begin{defn}\label{defn:invariant measure}
      Let $\cG$ be an ample groupoid with compact unit space. A probability measure $\mu$ on $\cG^{(0)}$ is said to be \textit{$\cG$-invariant} if
 $\mu(r(U)) = \mu(s(U))$ holds for every open bisection $U$. We denote by $M(\cG)$ the set of all invariant probability measures of $\cG$.
\end{defn}

Recall that the \emph{isotropy bundle} of a groupoid $\cG$ is defined to be $\cG' := \{g\in \cG \mid r(g) =  s(g)\}$. And for a given point $x \in \cG^{(0)}$, its isotropy group $\cG^x_x$ is defined to be $\{g \in \cG \mid r(g)=s(g) =x\}$

\begin{defn}\label{defn:free}
An \'etale groupoid is called \emph{principal} if $\cG' =\cG^{(0)}$. When the interior of $\cG'$ is $\cG^{(0)}$, we say that $\cG$ is \emph{effective}. It is called  \emph{essentially principal} if for any $\cG$-invariant probability measure, the set of points with trivial isotropy group has full measure.
\end{defn}

\begin{rem}\label{rem: ess prin}
   We remark that the essential principality in Definition \ref{defn:free} coincides with the notion of \textit{essential freeness} for \'{e}tale groupoids defined in \cite[Definition 3.3]{LZ24}, which means that 
   \[\mu(\{x\in s(B): r(Bx)=\{x\}\})=0\]
   holds for any $\mu\in M(\cG)$ and any precompact open bisection $B\subset \cG\setminus \cG^{(0)}$. Indeed, denote by $\mathcal{B}$ the collection of all precompact open bisections contained in $\cG\setminus \cG^{(0)}$ for simplicity. Then \'{e}taleness of $\cG$ implies that 
   \[\{x\in \cG^{(0)}: \cG^x_x=\{x\}\}^c=\bigcup_{B\in \mathcal{B}}\{x\in s(B): r(Bx)=\{x\}\}.\]
   This establishes the equivalence between the essential principality and the essential freeness for $\cG$.
   \end{rem}

We record a basic fact about essential principality and effectiveness (i.e., Matui's original essential principality in ,e.g., \cite{Renault80} and~\cite{Matui15}) for minimal \'{e}tale groupoids as follows.

\begin{prop}
Let $\cG$ be a minimal \'{e}tale topological groupoid. Suppose $\cG$ is essential principal. Then $\cG$ is effective.
\end{prop}
\begin{proof}
Denote by $D=\{x\in \cG^{(0)}: \cG^x_x=\{x\} \}$ for simplicity. Let $\mu\in M(\cG)$ and $O$ is a non-empty open set in $\cG^{(0)}$. Since $\cG$ is minimal and $\cG^{(0)}$ is compact, there are finitely many open bisections $U_1, \dots, U_n$ such that $\bigcup_{i=1}^ns(U_i)=\cG^{(0)}$ and $\bigcup_{i=1}^nr(U_i)\subset O$. This implies that $\mu(O)\geq 1/n>0$. In addition, note that $\mu(D)=1$ because $\cG$ is essential principal. But this implies that $D\cap O\neq \emptyset$ necessarily, i.e., $D$ is dense in $\cG^{(0)}$. 

Now, suppose the interior $\cG'^o\neq \cG^{(0)}$. Since $\cG^{(0)}$ is closed, there exists a non-empty open bisection $U\subset \cG'^o-\cG^{(0)}$. Define $V=s(U)$, which is a non-empty open set in $\cG^{(0)}$. Then $D\cap V=\emptyset$ holds by definition,  which is a contradiction. Thus $\cG'^o=\cG^{(0)}$ holds necessarily and thus $\cG$ is effective.
\end{proof}

\begin{defn}
A topological groupoid $\cG$ is called \emph{ample} if it is étale and its unit space $\cG^{(0)}$ is totally disconnected.
\end{defn}

\begin{defn}
An ample groupoid $\cG$ is called \emph{purely infinite} if for all compact  open subspaces $U,V \subseteq \cG^{(0)}$ with $V \neq \emptyset$, there exists a compact open bisection $\sigma \subseteq \cG$ such that $s(\sigma) = U$ and $r(\sigma) \subseteq V$.
\end{defn}

\begin{ex}[{cf.\ \cite[Definition I.1.7]{Renault80}}]
\label{ex:semi-direct-product}
Let $\cG$ be a topological groupoid and let $\Gamma$ be a discrete group acting on $\cG$ from the left via a homomorphism $\tau \colon \Gamma \to \mathrm{Aut}(\cG)$. We can then form the semi-direct product $\cG \rtimes_\tau \Gamma$ on the product space $\cG \times \Gamma$ with the following groupoid structure: $(g, \gamma) $ and $ (g', \gamma')$ are composable if and only if $g$ and $\tau(\gamma)(g')$ are composable, in which case the multiplication is given by
\[
(g, \gamma) \cdot (g', \gamma') = (g \cdot \tau(\gamma)(g'),\gamma \gamma').
\]
Inversion is given by
\[
(g,\gamma)^{-1} = (\tau(\gamma^{-1})(g^{-1}),\gamma^{-1})
\]
and the range and source maps are given by 
\[
r((g,\gamma)) = (r(g),1_\Gamma) \quad\text{and}\quad s((g,\gamma)) = (\tau(\gamma^{-1})(s(g)),1_\Gamma),
\]
where $1_\Gamma$ denotes the identity element of $\Gamma$. Note that the unit space of $\cG \rtimes_\tau \Gamma$ is the subspace $\cG^{(0)} \times \{1_\Gamma\}$ of the product space $\cG \times \Gamma$, so it is homeomorphic to $\cG^{(0)}$; in particular it is locally compact and Hausdorff. Moreover, it follows directly from the definitions that if $\cG$ is an étale (resp.\ ample) groupoid, then $\cG \rtimes_\tau \Gamma$ is also étale (resp.\ ample).

Finally, we note that there is a natural projection $\pi \colon \cG \rtimes_\tau H \to H$ given by $\pi(g, h) = h$.
\end{ex}

\begin{defn}
    A topological groupoid  is called  \emph{elementary} if it is principal  and compact.
\end{defn}

\begin{defn}
    Let $\cG$ be an ample groupoid with compact unit space. We say that a $\cK\subset \cG$ is an \emph{elementary subgroupoid} if $\cK$ is a compact open principal  subgroupoid of $\cG$ and $\cK^{(0)} = \cG^{(0)}$.
\end{defn}

\begin{rem}
    When $\cK$ is an elementary groupoid: (1) the topology of $\cK$ agrees with the relative topology from $\cK^{(0)}\times \cK^{(0)}$;(2) the number $\sup_{x\in \cK^{(0)}} |r^{-1}(x)|  < \infty$.
\end{rem}
 
\begin{defn}
    A topological groupoid is called \emph{approximately finite} if it can be written as an increasing union of open elementary subgroupoids.
\end{defn}

It is known that approximate finite topological groupoids can be represented by Bratteli diagrams. 

\begin{ex}
\label{ex:SFT}
Let $n\geq 2$ be a natural number. Consider $\{0,1,\cdots,n-1\}^{\bN}$, the set of infinite sequences in $\{0,1,\cdots,n-1\}$, equipped with the product topology. Note that $\{0,1,\cdots, n-1\}^{\bN}$ is homeomorphic to the Cantor space. Given an element $x \in \{0,1,\cdots,n-1\}^{\bN}$, we denote its $i$-th entry in the infinite sequence by $x_i$. We define a topological groupoid $\cE_n$ with morphisms:
\[
\cE_n := \{ (y,x) \in \{0,1,\cdots, n-1\}^{\bN}  \times \{0,1,\cdots,n-1\}^{\bN} \mid \exists~ k \in \bN, \text{ such that } x_i=y_i \text{ for any } i\geq k \}.
\]

The source and range maps are given by $s(y,x) = (x,x)$ and $r(y,x) =(y,y)$. The topology of $\cE_n$ comes from the subspace topology of $\cE_n^{(0)}\times \cE_n^{(0)}$. We note $\cE_n$ can be written as the union of the elementary  subgroupoids $\cE_n[k]$:
\[
\cE_n[k] := \{ (y,x) \in \{0,1,c\dots, n-1\}^{\bN}  \times \{0,1,\cdots,n-1\}^{\bN} \mid  x_i=y_i \text{ for any } i\geq k \}.
\]

In particular, $\cE_n$ is approximately finite.
\end{ex}

\begin{defn}
    Let $\cG$ be an ample groupoid with compact unit space. We say that $\cG$ is \emph{almost finite} if for any compact subset $C \subset \cG$ and $\varepsilon>0$ there exists an elementary subgroupoid $\mathcal{K} \subset \cG$ such that
$$
\frac{|C \mathcal{K} x - \mathcal{K} x|}{|\mathcal{K} x|}<\varepsilon
$$
for all $x \in \cG^{(0)}$.
\end{defn}

We end this section with the definition of amenable groupoids. Let us start with some setup. Let $C_c(\cG)$ be the set of continuous functions from $\cG$ to $\bR$ with compact support.  If $\cG$ is an \'etale groupoid, then a \emph{continuous system of probability measures} for $\cG$ is a system $\{\lambda^x\mid x\in \cG^{(0)}\}$ of Radon probability measures $\lambda^x$ on $\cG$ with the support of $\lambda^x$ contained in $\cG^x:=\{g\in \cG \mid r(g) =x\}$ such that for $f\in C_c(\cG)$, the function $x \mapsto \int f d\lambda^x$ is continuous. Note that since $\cG^x$ is discrete, a Radon probability measure $\lambda^x$ on $\cG^x$ amounts to a function $\lambda^x\colon \cG^x \to [0,\infty)$ with $\Sigma_{z\in \cG^{^x}} \lambda^x(z) = 1$.

An \emph{approximate invariant continuous mean} for $\cG$ is a net $\lambda_i$ of continuous systems of probability measure for $\cG$ such that the net $(M_i: \gamma \mapsto ||\lambda_i^{r(\gamma)}(\gamma \cdot) - \lambda_i^{s(\gamma)}(\cdot)||_1)_i$ of functions from $\cG$ to $[0,\infty$ has the property that $M_i\mid_K \to 0$ uniformly for every compact $K\subseteq \cG$.

\begin{defn}
    Let $\cG$ be an \'etale groupoid. We say that $\cG$ is (topologically) amenable if $\cG$ admits an approximate invariant continuous mean.
\end{defn}

\section{Examples arising from transformation groupoids}\label{sec:tranf}
In this section, we provide several (concrete) examples of transformation groupoids that are neither almost finite nor purely infinite. We begin with the following observation.
\begin{prop}\label{prop:tranf}
    Let $\alpha:\Gamma\curvearrowright X$ be a minimal topological dynamical system of a non-amenable countable discrete group on a compact Hausdorff space. Suppose there exists a $\Gamma$-invariant probability measure $\mu$ on $X$. Then the transformation groupoid $X\rtimes \Gamma$ is neither almost finite nor purely infinite.
\end{prop}
\begin{proof}
   The proof is straightforward. Because first, it is well-known that the almost finiteness of $X\rtimes \Gamma$ implies that $\Gamma$ is necessarily amenable. In addition, the pure infiniteness of $X\rtimes \Gamma$ entails that there exists no invariant probability measure $\mu$ on $X$ (see, e.g., \cite[Theorem 5.1]{Ma22}).
\end{proof}

It is worth to mention that such systems exists for all non-amenable groups by \cite[Theorem 1]{El21}. In the following, we provide explicit examples from non-amenable residually finite groups, which are known as odometers.

\begin{ex}\label{ex:odometer}
    Let $\Gamma$ be a residually finite non-amenable group and let $\{\Gamma_n: n\in \mathbb{N}\}$ be a decreasing sequence of finite index normal subgroups of $\Gamma$ such that $\bigcap_{n\in \mathbb{N}}\Gamma_n=\{e_\Gamma\}$.  Denote by $\alpha_n: \Gamma\curvearrowright \Gamma/\Gamma_n$ the natural action of $\Gamma$ on the quotient group $\Gamma/\Gamma_n$, equipped with the uniform probability measure $\mu_n$. Note that each $\mu_n$ is $\Gamma$-invariant. 
    Denote by $X=\lim\limits_{\leftarrow}\Gamma/\Gamma_n$ the inverse limit of $\Gamma/\Gamma_n$, on which there is a minimal action $\alpha$ induced by these transitive $\alpha_n$, together with a $\Gamma$-invariant probability invariant measure $\mu$ obtained from these measures $\mu_n$. Moreover, such the action $\alpha$ is free as  for $x=(s_n\Gamma_n)\in X$, its stabilizer \[\operatorname{stab}_\Gamma(x)=\bigcap _{n=1}^\infty s_n\Gamma_ns^{-1}_n=\bigcap_{n=1}^\infty\Gamma_n=\{e_\Gamma\}\]
    because all $\Gamma_n$ are normal. In addition, such an action has dynamical comparison (see, e.g., \cite[Lemma 2.4]{Joseph24}). As a summary, these dynamical systems yield minimal principal transformation groupoids satisfying the comparison but are neither almost finite nor purely infinite by Proposition \ref{prop:tranf}. 
\end{ex}

\begin{rem}
We remark that the transformation groupoids constructed from Proposition \ref{prop:tranf} are never topologically amenable. This is because it is known that the topologically amenability, together with the existence of an invariant probability measure will imply that the acting group $\Gamma$ is amenable. On the other hand,
it was constructed in \cite{Joseph24} a topologically amenable transformation groupoid $X\rtimes \Gamma$ with a unique invariant ergodic measure $\mu$  that are neither almost finite nor purely infinite (see also \cite{HirshbergWu24}). But such groupoids are only effective but not even essentially principal. 
Finally, we remark that so far there have been no known counterexample to Matui's question that does not arise from transformation groupoids.  
\end{rem}

In the following sections, we will provide the first examples of minimal effective non-transformation groupoids that are neither almost finite nor purely infinite.  Such groupoids can be additionally made to be essentially principal in some cases.

\section{Twisted topological groupoids and their basic properties}\label{section:basi-prop}

In this section, we recall the twisted topological groupoid construction of Palmer and Wu \cite[Remark 2.15]{PalmerWu25} and study its basic properties. \textbf{In this section, we always assume that our groupoid is an \'etale groupoid with compact unit space. }

Let $\cG$ be an \'etale groupoid. We can consider the infinite product set $\prod_T \cG$. We view an element $f\in \prod_T \cG$ as a map $f \colon T\to \cG$. 

We introduce the following notation.  Let $\{t_1, \dots ,t_k\}$ be a finite subset of $T$ and let $A_1$, $A_2$, $\dots$, $A_k$ be subsets of ${\cG}$, we define a subset $\mathcal{O}(A_1, \dots, A_k; t_1, \dots, t_k)$ of $\prod_T \cG$ by
$$\{f\in \prod_T \cG \mid f(t_i) \in A_i \text{ for } i\in \{1,\dots, k\} \text{, and } f(t) \in \cG^{(0)} \text{ for }  t \in T-\{t_1, \dots, t_k\}\}.$$

Let $\mathcal{O}$ be the collection of all open sets of $\cG$. Let us first recall the following construction from \cite[Remark 2.15]{PalmerWu25}. 
\begin{defn}
\label{defn:groupoid-svg}
Let $T$ be a countable set, and $\cG$ be an \'etale groupoid with compact unit space.  The groupoid $T\cG$ is defined as follows:
\begin{enumerate}
\item Unit space: $\prod_T \cG^{(0)}$.
\item Morphism space: $\{ f \in \prod_T \cG \mid f(t) \in \cG^{(0)} \text{ for all but finitely many } t\in T\}$.
\item Topological basis: $T\mathcal{O}:=\{\mathcal{O}(O_1,\dots,O_k;t_1,\dots,t_k)|O_i\in \mathcal{O},\text{for }i\in\{1,\dots,k\},\{t_1,\dots,t_k\}\subset T\}$.
\end{enumerate}

\end{defn}

\begin{rem}\label{unit space of TG}
    As a set, the unit space of $T\cG$ is just the infinite product $\prod_T\cG^{(0)}$. After we equip $T\cG$ with the topological basis $T\mathcal{O}$, the induced topology on the unit space $(T\cG)^{(0)}$ coincides with the product topology on $\prod_T\cG^{(0)}$. Since we assumed that the unit space $\cG^{(0)}$ is compact, the unit space $T\mathcal{G}^{(0)}$ is also compact.  Moreover, if $\cG^{(0)}$ is Hausdorff (resp. second countable, compact, totally disconnected), so is $(T\cG)^{(0)}$. Recall that by definition the unit space of a topological groupoid needs to be locally compact. This forces us to assume that the unit space of $\cG$ is compact.
\end{rem}

Let us prove that this is a well-defined topological groupoid.

\begin{lem}\label{lem:top}
    $T\cO$ forms a topological basis.
\end{lem}

\begin{proof}
It is clear that every element of $T\cG$ lies in some element of $T\cO$. It remains to show that if two elements in $T\cO$ has nonempty intersection $U$, and $x\in U$, then one can find an element $V \in T\cO$ such that $V\subset U$.

Let $\mathcal{O}(O_1, \dots, O_k; t_1, \dots, t_k)$ and $ \mathcal{O}(O_1^{\prime} , \dots,  O_k^{\prime}; t'_1, \dots, t'_k)$ be two elements in  $T\cO$. Since we assume that $\cG$ is \'etale, $\cG^{(0)}$ is an open subset of $\cG$, that is, $\cG^{(0)}\in\cO$.  Note now that for any $t \not\in \{t_1,\cdots, t_k\}$,  
$$\mathcal{O}(O_1, \dots, O_k; t_1, \dots, t_k) = \mathcal{O}(O_1, \dots, O_k, \cG^{{0}}; t_1, \dots, t_k, t).$$ With this we could assume that $t_i = t_i'$ for $1\leq i \leq k$. But then 
\[    \mathcal{O}(O_1, \dots, O_k; t_1, \dots, t_k)\cap \mathcal{O}(O_1^{\prime} , \dots,  O_k^{\prime}; t_1, \dots, t_k)=\mathcal{O}(O_1\cap O_1^{\prime} , \dots, O_k \cap O_k^{\prime}; t_1, \dots, t_k),\]
which clearly still lies in $T\cO$. 
\end{proof}

 The source map $s^T$ of $T\cG$ is given by $s^T(f)(t):=s(f(t))$ for $t \in T$ and the range map $r^T$ is given by $r^T(f)(t):=r(f(t))$ for $t \in T$. When there is no ambiguity, we may abbreviate $s^T$ and $r^T$ by $s$ and $r$. For two elements $f_1$ and $f_2$ in ${T\cG}$ such that $s(f_1)=r(f_2)$, the multiplication is defined by $(f_1\cdot f_2)(t):=f_1(t)\cdot f_2(t)$ for $t \in T$ and the inverse of $f_1$ is defined by $(f_1^{-1})(t):=f_1(
t)^{-1}$ for $t \in T$.

\begin{lem}\label{lem:s-r-i-cont}
  The multiplication and inversion maps of $T\cG$ are continuous.    
\end{lem}
\begin{proof}
    We first show that the inversion map is continuous. Let $\mathcal{I}$ be the inversion map of $\cG$ and $\mathcal{I}_T$ be the inversion map of $T\cG$. For any open set $\mathcal{O}(O_1, \dots, O_k; t_1, \dots, t_k)$ in $T\cG$, we have 
    \[\mathcal{I}_T^{-1}(\mathcal{O}(O_1, \dots, O_k; t_1, \dots, t_k))=\mathcal{O}(\mathcal{I}^{-1}(O_1), \dots, \mathcal{I}^{-1}(O_k); t_1, \dots, t_k).\] For each $i\in \{1,\dots, k\}$, we know that $\mathcal{I}^{-1}(O_i)$ is open in $\cG$. Hence, $\mathcal{I}_T^{-1}(\mathcal{O}(O_1, \dots, O_k; t_1, \dots, t_k))$ is open and $\mathcal{I}_T$ is continuous.

Let $\mathcal{M}$ be the multiplication map of $\cG$ and $\mathcal{M}_T$ be the multiplication map of $T\cG$. For any open set $\mathcal{O}(O_1, \dots, O_k; t_1, \dots, t_k)$ in $T\cG$, we show that $\mathcal{M}_T^{-1}(\mathcal{O}(O_1, \dots, O_k; t_1, \dots, t_k))$ is open in $T\cG \ _{\mathrm{s}}\times_{\mathrm{ r}} T\cG$. Let $(f_1,f_2)$ be an arbitrary element in the preimage. Since $f_1$ and $f_2$ are composable, there exists a finite subset $T^{\prime}$ of $T$ that contains $\{t_1,\dots,t_k\}$ such that for every $t \in T-T^{\prime}$, $f_1(t)=f_2(t) \in 
\cG^{(0)}$. For $t^{\prime} \in T^{\prime}$, we have that $f_1(t^{\prime})\cdot f_2(t^{\prime})$ lies in $O_i$ for some $i\in \{1,\dots, k\}$, or lies in $\cG^{(0)}$, and both of them are open. Denote by $U_{t^\prime}$ the open set $O_i$ or $\cG^{(0)}$, and we know that $\mathcal{M}^{-1}(U_{t^\prime})$ is open in $\cG \ _{\mathrm{s}}\times_{\mathrm{ r}} \cG$, so there exists an open set $U_{t^\prime ,1}$ of $\cG$ containing $f_1(t^{\prime})$ and an open set $U_{t^\prime ,2}$ of $\cG$ containing $f_2(t^{\prime})$ such that $(U_{t^\prime ,1} \times U_{t^\prime ,2}) \cap (\cG \ _{\mathrm{s}}\times_{\mathrm{ r}} \cG) \subseteq \mathcal{M}^{-1}(U_{t^\prime})$. Then using the open sets $U_{t^\prime ,1}$ and $U_{t^\prime ,2}$ for every $t^\prime \in T^\prime$, we can construct an open neighborhood 
\[(T\cG_s\times_rT\cG)\cap \cO(U_{t', 1}; t'\in T')\times \cO(U_{t', 2}; t'\in T')\]
of $(f_1,f_2)$ that contained in $ \mathcal{M}_T^{-1}(\mathcal{O}(O_1, \dots, O_k; t_1, \dots, t_k))$. Thus, $\mathcal{M}_T^{-1}(\mathcal{O}(O_1, \dots, O_k; t_1, \dots, t_k))$ is open and the multiplication map is continuous.
\end{proof}
With Lemma \ref{lem:top} and Lemma \ref{lem:s-r-i-cont}, we have that $T\cG$ is a well-defined topological groupoid. And the new groupoid $T\cG$ inherits many properties from $\cG$.

\begin{prop}\label{from G to TG}
    Let $\cG$ be an \'etale groupoid with compact unit space and $T$ be a countable set, then:
\begin{enumerate}
    \item  $T\cG$ is \'etale.

    \item  If $\cG$ is ample, so is $T\cG$.

    \item  If $\cG$ is minimal, so is $T\cG$.

    \item \label{prop:GtoTG-p-ep} If $\cG$ is  principal (resp. effective), so is $T\cG$.

    \item  If $\cG$ is an almost finite, so is $T\cG$.

    \item   If $\cG$ is not topologically amenable, so is $T\cG$.
    \end{enumerate}
\end{prop}

\begin{proof}
    (1) Since $T\cO$ forms a topological basis, it's easy to see that $T\cG$ and $T\cG^{(0)}$ are also second countable, Hausdorff and locally compact. To show $T\cG$ is  \'etale, it suffices to prove that the source and range maps $s^T$ and $r^T$ are local homeomorphisms. Pick any point $f \in T\cG$. By definition for all but finitely many $t\in T$, one has $f(t)\in \cG^{(0)}$. List them as $t_1,\cdots,t_k$. For each $t_i$, since $\cG$ is \'etale, there is a neighborhood $U_i$ for $f(t_i)$ such that $s\mid_{U_i}$ is a homeomorphism from $U_i$ to $s(U_i)$. This implies that $s^T$ restricted to $\cO((U_1,\dots ,U_k;t_1,\dots,t_k))$ is a homeomorphism. Finally, let $B_1, \dots, B_k$ be open bisections in $\cG$ and observe that
    \[s^T(\cO(B_1,\dots, B_k;t_1,\dots,t_k))=\cO(s(B_1),\dots, s(B_k);t_1,\dots,t_k);\]
    which implies that $s^T$ is an open map. Therefore, the map $s^T$ is a local homeomorphism. The case of range map is similar.

    (2) Note that a groupoid is ample if and only if it is \'etale and its unit space is totally disconnected. then by   (1) and Remark \ref{unit space of TG},  $T\cG$ is also ample.

    (3) We need to show that for every $x \in (T\cG)^{(0)}$, the orbit $r(T\cG  x)$  is dense in $(T\cG)^{(0)}$. For an arbitrary open set $\mathcal{O}(U_1,\dots ,U_k;t_1,\dots,t_k)\subset (T\cG)^{(0)}$, where $U_1,\dots ,U_k$ are open subsets of $\cG^{(0)}$, it suffices to find an element $f\in T\cG$ such that $r(fx)\in \mathcal{O}(U_1,\dots ,U_k;t_1,\dots,t_k)$. Since $\cG$ is minimal, for each $x(t_i)\in\cG^{(0)}$ and each open subset $U_i\subset\cG^{(0)}$, there exists a $g_i\in\cG$ such that $s(g_i)=x(t_i)$ and $r(g_i)\in U_i$ for all $i\in\{1,\dots,k\}$.  Let $f\in T\cG$ with $f(t_i)=g_i$ for $i \in \{1,\dots,k\}$ and $f(t)=x(t)$ for $t\in T-\{1,\dots,k\}$.  Then we have $s(f)=x$ and $r(fx)=r(f) \in \mathcal{O}(U_1,\dots ,U_k;t_1,\dots,t_k)$.

    (4) If $\cG$ is principal, then for every $f\in T\cG$ with $r(f)=s(f)$, we have $r(f(t))=s(f(t)) \in \cG$ for all $t\in T$. By $\cG$ being principal, we have $f(t)=r(f(t))=s(f(t)) \in \cG^{(0)}$ for $t\in T$, which implies $f\in (T\cG)^{(0)}$.
    
Now assume that the groupoid $\cG$ is effective. By definition,  for every $g \in \cG^\prime-\cG^{(0)}$ and every neighborhood $O_g$ of $g$, there exists an element $\bar{g} \in O_g$ such that $\bar{g} \notin \cG^\prime$. To show that $T\cG$ is effective, it suffices to prove that for every $f\in (T\cG)^\prime-(T\cG)^{(0)}$ and every neighborhood $U_f$ of $f$, one can find $\bar{f} \in U_f$ such that $\bar{f} \notin (T\cG)^\prime$. Note that $T\mathcal{O}$ is a topological basis, for each $U_f$ there exists an open subset $\mathcal{O}(O_1,\dots ,O_k;t_1,\dots,t_k)$ in $T\mathcal{O}$ that is contained in $U_f$ and contains $f$. Since $f\in (T\cG)^\prime-(T\cG)^{(0)}$, without loss of generality, we can assume that $O_1$ is not contained in $\cG^{(0)}$, and $f(t_1)\in \cG^\prime - \cG^{(0)}$. Then for the neighborhood ${O}_1$ of $f(t_1)$, using that $\cG$ is effective, there exists a $g_1 \in {O}_1$ such that $g_1\notin \cG^\prime$. Let $\bar{f}\in T\cG$ with $\bar{f}(t_1)=g_1$ and $\bar{f}(t)=f(t)$ for $t \in T-\{t_1\}$. Then $\bar{f}\in U_f$ and $\bar{f} \notin (T\cG)^\prime$ since $\bar{f}(t_1)=g_1 \notin \cG^\prime$.

(5) Let $K\subset T\cG$ be a compact set and $\varepsilon>0$. Without loss of generality, one may assume $(T\cG)^{(0)}\subset K$. First, for $K$ there exists a finite set $F\subset T$ such that $K\subset \prod_{i\in F}\cG\times\prod_{i\in T-F}\cG^{(0)}$. Then for each $i\in F$, define $K_i=\pi_i(K)$, where $\pi_i: T\cG\to \cG$ is the projection map on the $i$-th coordinate. Note that $\cG^{(0)}\subset K_i$ holds for any $i\in F$. Moreover, define $K'=\prod_{i\in F}K_i\times \prod_{i\in T-F}\cG^{(0)}$ and choose $\delta>0$ such that $\delta^{|F|}<\varepsilon$.

    Then because $\cG$ is assumed to be almost finite, then for the compact $K_i\subset \cG$ and $\delta$, there exists an elementary subgroupoid $\cH_i\subset \cG$ satisfying 
    \[|K_i\cH_i u_i-\cH_iu_i|<\delta|\cH_iu_i|\]
    for any $u_i\in \cH^{(0)}_i=\cG^{(0)}$. Define $\cH=\prod_{i\in F}\cH_i\times \prod_{i\in T-F}\cG^{(0)}$, which is still an elementary groupoid by definition. Then for any $u=(u_i)_{i\in T}\in (T\cG)^{(0)}=\prod_{i\in T}\cG^{(0)}$, one has
    \[|K\cH u-\cH u|\leq |K'\cH u-\cH u|=\prod_{i\in F}|K_i\cH_i u_i-\cH_iu_i|<\prod_{i\in F} \delta|\cH_iu_i|\leq \varepsilon |\cH u|.\]
    This implies that $T\cG$ is almost finite.

    (6) For the last claim, choose a $u_0\in \cG^{(0)}$ and fix a $t_0\in T$, then there exists a natural embedding from $\cG\to T\cG$ by $\gamma\mapsto f$ such that $f(t_0)=\gamma$ and $f(t)=u_0$ for any $t\neq t_0$. This implies that $\cG$ can be viewed as a closed subgroupoid of $T\cG$. Now, suppose $T\cG$ is amenable. then so is $\cG$ by \cite[Proposition 4.1.14]{Sim17}. Therefore, the claim follows.
\end{proof}

    Let us now define the twisted topological groupoid. Assume that  $\Gamma$ is a countable group acting  on a countable set $T$. This action induces an action of $\Gamma$ on $T\cG$: given $\gamma \in \Gamma$, let $\tau_\gamma : T\cG \rightarrow T\cG$ be the isomorphism of groupoid $T\cG$ defined by:
    \[[\tau_\gamma(f)](t):=f(\tau^{-1}t), \text{ for } f\in T\cG \text{ and } t \in T.\]
    And we denote the action by $\gamma \cdot f:=\tau_\gamma(f)$, for $f\in T\cG$ and $\gamma \in \Gamma$. With this action, one can consider the semi-direct product $T\cG\rtimes_\tau \Gamma$ (cf. Example \ref{ex:semi-direct-product}). We shall abbreviate $T\cG\rtimes_\tau \Gamma$ by $T\cG\rtimes \Gamma$ when the action is clear from the context.
    
    \begin{defn}
    Let $\Gamma$ be a countable group acting on a countable set $T$ and let $\cG$ be an \'etale groupoid with compact unit space. Then the \emph{twisted topological groupoid} is defined by $T\cG\rtimes \Gamma$.
\end{defn}
    
    Recall that the source map $s_{\Gamma}$ of $T\cG\rtimes \Gamma$  is defined by $s_{\Gamma}(f,\gamma)=(\gamma^{-1}\cdot s(f),1_{\Gamma})$, and  the range map $r_{\Gamma}(f,\gamma)=(r(f),1_{\Gamma})$. 
The topology on $(T\cG)\rtimes \Gamma$ is just the product topology, and so the topological basis consists of $\mathcal{O}(O_1,\dots ,O_k;t_1,\dots,t_k)\times\{\gamma\}$ for all $\mathcal{O}(O_1,\dots ,O_k;t_1,\dots,t_k)\in T\mathcal{O}$ and $\gamma\in \Gamma$. 

\begin{obs} \label{TG is a subgroupoid}
    The groupoid $T\cG$ can be viewed as a subgroupoid of $T\cG\rtimes \Gamma$ by identifying $f\in T\cG$ with $(f, 1_{\Gamma})\in T\cG\rtimes \Gamma$. And the unit space of $T\cG\rtimes \Gamma$ is $T\cG^{(0)}\times \{1_\Gamma\}$ which is the same as the unit space of $T\cG$.
\end{obs}

Just as in the case of $T\cG$, the twisted topological groupoid $T\cG\rtimes \Gamma$  also inherits many properties from $\cG$.

\begin{prop}\label{from G to TGxGamma}
    Let $\Gamma$ be a group acting  on a countable set $T$ and let $\cG$ be an \'etale groupoid with compact unit space, then
    \begin{enumerate}
        \item $T\cG\rtimes \Gamma$ is \'etale.
        \item  If $\cG$ is ample, so is $T\cG\rtimes \Gamma$.
        \item If $\cG$ is minimal, so is $T\cG\rtimes \Gamma$.
        \item \label{prop:tgg-case-ep} Suppose that $\Gamma$ acts faithfully on $T$. If $\cG$ is effective and $\cG^{(0)}$ has no isolated point, then $T\cG\rtimes \Gamma$ is effective. 
        \item \label{prop:tgg-case-ep2}Suppose that $\Gamma$ is an infinite group acting freely on $T$.  If $\cG$ is effective and $\cG^{(0)}$ has at least two points, then $T\cG\rtimes \Gamma$ is effective.
        
    \end{enumerate}    
\end{prop}  
\begin{proof}

    The first 3 properties follows immediately from Proposition \ref{from G to TG} and the definition of the semi-direct product, see Example \ref{ex:semi-direct-product}.

    Let us prove (\ref{prop:tgg-case-ep}) now. Since $\cG$ is \'etale, there exists $\mathcal{B}$ consisting of open bisections of $\cG$ which forms a topological basis for $\cG$. Let $T\mathcal{B}_\Gamma$ be the collection of all open subsets of $T\cG\rtimes \Gamma$ that have the form $\mathcal{O}(B_1,\dots ,B_k;t_1,\dots,t_k)\times\{\gamma\}$ where $B_i\in \mathcal{B}$ for $i=1,\dots,k$ and $\gamma \in \Gamma$.   Let $(f,\gamma) \in (T\cG\rtimes \Gamma)^\prime-(T\cG\rtimes \Gamma)^{(0)}$, and let $\mathcal{O}(B_1,\dots,B_k;t_1,\dots,t_k)\times\{\gamma\}$ be an arbitrary open set in $T\mathcal{B}_\Gamma$ that contains $(f,\gamma)$. Without loss of generality, we can assume that $\gamma \neq 1_\Gamma$. Since when $\gamma=1_\Gamma$, by Observation \ref{TG is a subgroupoid} and Proposition \ref{from G to TG}, we can find an element $(\bar{f},1_\Gamma) \in \mathcal{O}(B_1,\dots,B_k;t_1,\dots,t_k)\times\{1_\Gamma\}$ such that $(\bar{f},1_\Gamma) \notin (T\cG\rtimes \Gamma)^\prime$.

    Now, since $\gamma \neq1_\Gamma$, and $\Gamma$ acts on $T$ faithfully. Then there exists some $t_0\in T$ such that $\gamma\cdot t_0 \neq t_0$. And there exists a clopen bisection $U_0$ contained in $\cG$ such that $f(t_0)\in U_0$, where $U_0=B_i$ for some $i \in \{1,2,\dots,k\}$, or $U_0=\cG^{(0)}$. Since $\cG^{(0)}$ has no isolated point and $U_0$ is an open bisection, for $f(t_0)\in U_0$, there exists another element ${g}_0\in U_0$ such that $r(g_0) \neq r(f)(t_0)$. Define $\tilde{f}\in T\cG$ by $\tilde{f}(t_0):=g_0$ and $\tilde{f}(t):=f(t)$ for $t\in T-\{t_0\}$. Then $\tilde{f} \in \mathcal{O}(B_1,\dots,B_k;t_1,\dots,t_k)$.
    Note that $(f,\gamma) \in (T\cG\rtimes \Gamma)^\prime$ implies that $s_\Gamma(f,\gamma)=r_\Gamma(f,\gamma)$, then we have $s_\Gamma(f,\gamma)(t)=s(f)(\gamma\cdot t)=r(f)(t)=r_\Gamma(f,\gamma)(t)$ for all $t\in T$. Using $\tilde{f}(\gamma\cdot t_0)=f(\gamma\cdot t_0)$ and $\tilde{f}( t_0)=g_0$, we have 
    $$s_\Gamma(\tilde{f},\gamma)(t_0)=s(\tilde{f})(\gamma\cdot t_0)=s(f)(\gamma\cdot t_0)=r_\Gamma(f,\gamma)(t_0)=r(f)(t_0)$$ 
    and
    $$r_\Gamma(\tilde{f},\gamma)(t_0)=r(\tilde{f})(t_0)=r(g_0)$$

Since $g_0$ is an element such that $r(g_0)\neq r(f)(t_0)$, then $s_\Gamma(\tilde{f},\gamma)(t_0) \neq r_\Gamma(\tilde{f},\gamma)(t_0)$ and $(\tilde{f},\gamma) \in \mathcal{O}(B_1,\dots,B_k;t_1,\dots,t_k)\times\{\gamma\}$ such that $(\tilde{f},\gamma) \notin (T\cG\rtimes \Gamma)^\prime$. This implies that $T\cG\rtimes \Gamma$ is effective.

Finally, let us prove (\ref{prop:tgg-case-ep2}). Without loss of generality, we consider $(f,\gamma)\in (T\cG\rtimes\Gamma)^\prime$ with $\gamma \neq 1_\Gamma$. Let $\mathcal{O}(U_1,\dots,U_k;t_1,\dots,t_k)\times\{\gamma\}$ be an arbitrary neighborhood of $(f,\gamma)$, where $U_i$ are open bisections. Note that if $\tilde{t}\in T-(\{\gamma^{-1}\cdot t_1,\dots,\gamma^{-1}\cdot t_k\}\cup\{t_1,\dots,t_k\})$, then both $f(\tilde{t})$ and $f(\gamma\cdot \tilde{t})$ lie in $\cG^{(0)}$. Since $\Gamma$ is an infinite group and the action is free, such $\tilde{t}$ always exists and $\gamma\cdot\tilde{t}\neq\tilde{t}$. Applying the source and range maps, we have 
$$s_\Gamma(f,\gamma)(\tilde{t})=s(f)(\gamma\tilde{t})=f(\gamma\cdot\tilde{t})$$
and
$$r_\Gamma(f,\gamma)(\tilde{t})=r(f)(\tilde{t})=f(\tilde{t}).$$

By $(f,\gamma)\in (T\cG\rtimes\Gamma)^\prime$, we have $f(\tilde{t})=f(\gamma\cdot\tilde{t})$. From the free action, we have that $\gamma\cdot \tilde{t}\neq\tilde{t}$. Since $\cG^{(0)}$ has at least two points, let $g\in \cG^{(0)}$ such that $g\neq f(\tilde{t})$. Define $\tilde{f}$ in $\mathcal{O}(U_1,\dots,U_k;t_1,\dots,t_k)$ by $\tilde{f}(t)=f(t)$ for $t\neq\tilde{t}$, and $\tilde{f}(\tilde{t})=g$. Then $s(\tilde{f},\gamma)(\tilde{t})=f(\tilde{t})$ and $s(\tilde{f},\gamma)(\tilde{t})=g$. Hence, there exists $(\tilde{f},\gamma)\in \mathcal{O}(U_1,\dots,U_k;t_1,\dots,t_k)\times\{\gamma\}$ such that $(\tilde{f},\gamma)\notin (T\cG\rtimes\Gamma)^\prime$. This implies that $T\cG\rtimes \Gamma$ is effective.
\end{proof}

\begin{rem}
    In general, the assumption in Proposition \ref{from G to TGxGamma} (\ref{prop:tgg-case-ep}) that the unit space $\cG^{(0)}$ has no isolated point cannot be omitted. Suppose that $p\in \cG^{(0)}$ is an isolated point and suppose that there exists an involution $\gamma\in \Gamma$ such that $\gamma(t_1)=t_2$ and $\gamma(t_2)=t_1$ for different $t_1$, $t_2$ in $T$, and $\gamma(t)=t$ for $t\in T-\{t_1,t_2\}$. Then we have the nontrivial neighborhood $\cO(\{p\},\{p\};t_1,t_2)\times\{\gamma\}\subseteq(T\cG\rtimes\Gamma)^\prime$.  In this case, $T\cG\rtimes \Gamma$ is not effective.
\end{rem}
\begin{rem}\label{never principal}
    Unlike Proposition \ref{from G to TG} (\ref{prop:GtoTG-p-ep}), the twisted topological groupoid $T\cG\rtimes \Gamma$ is in general not principal even if $\cG$ is principal . In fact, pick $x\in \cG^{(0)}$, let $\bar{x}$ be the element of $T\cG$ defined by $\bar{x}(t)=x$ for any $t\in T$. Then $(\bar{x},\gamma)$ lies in the isotropy bundle for any $\gamma \in \Gamma$.
\end{rem}

\section{Proof of the main theorem}   

In this section, we construct minimal ample effective non-transformation groupoids that are neither almost finite nor purely infinite. Such groupoids can be made essentially principal  in the unique invariant probability measure case. See Example \ref{ex:non a.f. and non p.i.}.

Let $\Gamma$ be a non-amenable group, and let the set $T$ also be $\Gamma$ with $\Gamma$ acting freely on it by left multiplication. For arbitrary  ample groupoid $\cG$ with compact unit space, we consider the twisted topological groupoid ${\Gamma}\cG\rtimes \Gamma$. We proceed to show that this groupoid is not almost finite. 

Recall that the amenability of a group can be characterized by the Følner condition, for a non-amenable group $\Gamma$, we have the following counting lemma. See for example \cite[Proposition 4.7.1]{Ceccherini-SCoornaert23}.

\begin{lem}\label{counting in nonamenable group}
    Let $\Gamma$ be any non-amenable group. Then there is a finite 
    subset $B$ of $\Gamma$ and some $\delta>0$ such that for every finite subset $K\subset \Gamma$, one has $|BK-K|\geq \delta |K|$.
\end{lem}

\begin{thm}\label{non alomst finiteness}
    Let $\cG$ be an ample groupoid $\cG$ with compact unit space, $\Gamma$ a non-amenable group. Then the twisted topological groupoid ${\Gamma}\cG\rtimes \Gamma$ is an ample groupoid with compact unit space that is not almost finite.
\end{thm}
\begin{proof}
By Proposition \ref{from G to TGxGamma}, it suffices now to show that ${\Gamma}\cG\rtimes \Gamma$ is not almost finite. As mentioned in Observation \ref{TG is a subgroupoid}, the unit space $({\Gamma}\cG\rtimes \Gamma)^{(0)}$ can be viewed as $({\Gamma}\cG)^{(0)}$ which is equal to the infinite product $\Pi_{\Gamma}\cG^{(0)}$. We consider the finite set $B$ as in Lemma \ref{counting in nonamenable group}. Consider the compact subset $C:=({\Gamma}\cG)^{(0)}\times B$ contained in ${\Gamma}\cG\rtimes \Gamma$. We will show that for any elementary subgroupoid $\mathcal{K}$ of ${\Gamma}\cG\rtimes \Gamma$ and any $x\in ({\Gamma}\cG\rtimes \Gamma)^{(0)}$, we have 
$$
\frac{|C \mathcal{K} x - \mathcal{K} x|}{|\mathcal{K} x|}\geq\delta.
$$

 
 We now work on $\Gamma\cG\times\Gamma$, which equals  $\Gamma\cG\rtimes \Gamma$ as a set. The group $\Gamma$ acts on it by diagonal action $\gamma_1\cdot(\kappa,\gamma)=(\gamma_1\cdot\kappa,\gamma_1\cdot\gamma)$ where $\gamma_1\in \Gamma$ and $(\kappa,\gamma)\in {\Gamma}\cG\times \Gamma$. This is a free action and for $(\kappa,\gamma)\in {\Gamma}\cG\times \Gamma$, we denote the orbit of $(\kappa,\gamma)$ by $O(\kappa,\gamma)$. Then we can divide the set ${\Gamma}\cG\times \Gamma$ into the union of orbits 
$${\Gamma}\cG\times \Gamma=\bigcup_{(\kappa,\gamma)\in {\Gamma}\cG\times \Gamma}O(\kappa,\gamma)$$
Two orbits are either the same or disjoint. One can view this as an equivalence relation on ${\Gamma}\cG\times \Gamma$ that $(\kappa,\gamma)\sim (\kappa_1,\gamma_1)$ if and only if they are in the same $\Gamma$-orbit. Then for any elementary subgroupoid $\cK$,  the equivalence relation $\sim$ restricts to the subset $\mathcal{K}x\subset {\Gamma}\cG\times \Gamma$, a sub-equivalence relation on $\mathcal{K}x$. Since $\mathcal{K}$ is elementary, $|\mathcal{K}x|<\infty$, and we can divide $\mathcal{K}x$ into disjoint union of finitely many equivalence classes, denoted by $\mathcal{K}x=\mathcal{K}_1\bigsqcup\mathcal{K}_2\cdots\bigsqcup\mathcal{K}_n$. Here we assume that $\mathcal{K}_i$ is non-empty, for $i\in\{1,\dots,n\}$ and by picking one element $(\kappa_i,\gamma_i)$ in each $\mathcal{K}_i$, we have $\mathcal{K}_i=\mathcal{K}x\cap O(\kappa_i,\gamma_i)$ and $\mathcal{K}x=\mathcal{K}x\cap(O(\kappa_1,\gamma_1)\bigsqcup O(\kappa_2,\gamma_2)\cdots\bigsqcup O(\kappa_n,\gamma_n))$.

The element in $C$ has the form $(\iota,\sigma)$ where $\iota:\Gamma\rightarrow\cG^{(0)}$ is a map from $\Gamma$ to $\cG^{(0)}$ and $\sigma\in B$. Applying the source map to $(\iota,\sigma)$, we have $s(\iota,\sigma)=(\sigma^{-1}\cdot\iota,1_\Gamma)$ where $1_\Gamma$ is the identity of $\Gamma$. Then for an arbitrary element $(\eta,1_\Gamma)\in ({\Gamma}\cG\rtimes \Gamma)^{(0)}$, there are exactly $|B|$ elements $(b\cdot\eta,b)$ for $b\in B$ in $C$ whose source is equal to $(\eta,1_\Gamma)$. And the composition of elements in $C$ and $\mathcal{K}$ coincides with the action of $\Gamma$ on ${\Gamma}\cG\times \Gamma$, i.e., \[(\iota,\sigma)\cdot(\kappa,\gamma)=(\iota\cdot(\sigma\cdot\kappa),\sigma\cdot \gamma)=(\sigma\cdot\kappa,\sigma\cdot \gamma)=\sigma\cdot(\kappa,\gamma)\] holds for $(\iota,\sigma)\in C$ and $(\kappa,\gamma)\in \mathcal{K}$. Let $(\bar{\kappa}_i,\bar{\gamma}_i)$ be an element in $\mathcal{K}_i$, then there are exactly $|B|$ elements in $C$ such that their sources all are equal to the range of $(\bar{\kappa}_i,\bar{\gamma}_i)$. In fact, one has $C\cdot(\bar{\kappa}_i,\bar{\gamma}_i) = \{(b\cdot\bar{\kappa}_i,b\cdot\bar{\gamma}_i)\mid b\in B\}$ which implies that $C\cdot(\bar{\kappa}_i,\bar{\gamma}_i) \subset O(\kappa_i,\gamma_i)$ and $C\cdot\mathcal{K}_i \subset O(\kappa_i,\gamma_i)$. Since for distinct $i$ and $j$ in $\{1,2,\dots,n\}$, we have $O(\kappa_i,\gamma_i)\cap O(\kappa_j,\gamma_j)=\emptyset$, then $C\cdot\mathcal{K}_i\cap \mathcal{K}_j=\emptyset$ and $C\mathcal{K}_i-\mathcal{K}x=C\mathcal{K}_i-\mathcal{K}_i$.

Since the action of $\Gamma$ on itself is free, we have a canonical way to identify the orbit $O(\kappa,\gamma)$ with the group $\Gamma$ by identifying $(\kappa,\gamma)\in O(\kappa_i,\gamma_i)$ with $\gamma\in \Gamma$. For each $i\in\{1,\dots,n\}$, $\mathcal{K}_i=\mathcal{K}x\cap O(\kappa_i,\gamma_i)$ is one to one correspondence to a finite subset denoted by $K_i\subset \Gamma$, and $C\mathcal{K}_i$ is one to one correspondence to $BK_i\subset \Gamma$. So we have $$|C\mathcal{K}_i-\mathcal{K}x|=|C\mathcal{K}_i-\mathcal{K}_i|=|BK_i-K_i|$$

Applying lemma \ref{counting in nonamenable group}, we have 
\begin{align*}
    |C\mathcal{K}x-\mathcal{K}x|&=|C(\mathcal{K}_1\bigsqcup\mathcal{K}_2\cdots\bigsqcup\mathcal{K}_n)-\mathcal{K}x|\\
    &=|(C\mathcal{K}_1-\mathcal{K}x)\bigsqcup(C\mathcal{K}_2-\mathcal{K}x)\cdots\bigsqcup(C\mathcal{K}_n-\mathcal{K}x)|\\
    &=|(C\mathcal{K}_1-\mathcal{K}_1)\bigsqcup(C\mathcal{K}_2-\mathcal{K}_2)\cdots\bigsqcup(C\mathcal{K}_n-\mathcal{K}_n)|\\
    &= |C\mathcal{K}_1-\mathcal{K}_1|+|C\mathcal{K}_2-\mathcal{K}_2|+\cdots+|C\mathcal{K}_n-\mathcal{K}_n|\\
    &= |BK_1-K_1|+|BK_2-K_2|+\cdots+|BK_n-K_n|\\
    &\geq \delta(|K_1|+|K_2|+\cdots+|K_n|)\\
    &=\delta(|\mathcal{K}_1|+|\mathcal{K}_2|+\cdots+|\mathcal{K}_n|)\\
    &=\delta|\mathcal{K}x|
\end{align*}

\end{proof}

Next, we consider the pure infiniteness of the twisted topological groupoids. We will show that if we let the groupoid $\cG$ be an almost finite groupoid, then the twisted topological groupoid can never be  purely infinite.

Recall that when a groupoid has invariant probability measure, then by definition it can not be purely infinite. But for an almost finite groupoid, there always exists an invariant measure.

\begin{lem}\label{a.f. groupoid has invariant mesure}
    {\cite[Lemma 6.5.]{Matui2012}} Let $\cG$ be an almost finite groupoid. Then there exists a $\cG$-invariant  measure $\mu$ on $\cG^{(0)}$. 
\end{lem}

With this Lemma, we can prove the following Theorem.

\begin{thm}\label{non-p.i.}
    Let $\cG$ be an almost finite groupoid. Let $\Gamma$ be a group acting faithfully on a countable set $T$. Then the twisted topological groupoid $T\cG\rtimes\Gamma$ has a $(T\cG\rtimes\Gamma)$-invariant measure. In particular, the  groupoid $T\cG\rtimes\Gamma$ is not purely infinite.
\end{thm}
\begin{proof}
    By Lemma \ref{a.f. groupoid has invariant mesure}, there exists a $\cG$-invariant measure $\mu$ on $\cG^{(0)}$. Since the groupoid $\cG$ is also an ample groupoid, there exists $\cB$ consisting of clopen bisections that form a topological basis for $\cG$.  And for every $B\in \cB$, we have $\mu(s(B))=\mu(r(B))$.

    As we explain in Remark \ref{unit space of TG} and Observation \ref{TG is a subgroupoid}, The subspace topology on $(T\mathcal{G}\rtimes\Gamma)^{(0)}=\Pi_T \cG^{(0)} \times\{1_\Gamma\}$ is homeomorphic to the product topology on $\Pi_T \cG^{(0)}$. Denote $\mu_T$ by the product measure on $\Pi_T \cG^{(0)}$ getting from $\mu$. We show that $\mu_T$ is an $(T\cG\rtimes\Gamma)$-invariant measure for the twisted topological groupoid $T\cG\rtimes\Gamma$.

    Let $T\cG_\Gamma$ be the collection of clopen bisections of the form $\cO(B_1,B_2,\dots,B_k;t_1,\cdots,t_k)\times\{\gamma\}$ for $B_i\in \cB$ and $\gamma \in \Gamma$. Then $T\cG_\Gamma$ also forms a topological basis for $T\cG\rtimes\Gamma$. Applying the source map $s_\Gamma$ and the range map $r_\Gamma$ to one $\cO(B_1,B_2,\dots,B_k;t_1,\cdots,t_k)\times\{\gamma\}$ in $T\cG_\Gamma$, we have that 
$$s_\Gamma(\mathcal{O}(B_1,\dots,B_k;t_1,\dots,t_k)\times\{\gamma\})=\mathcal{O}(s(B_1),\dots,s(B_k);\gamma^{-1} \cdot  t_1,\dots,\gamma^{-1}\cdot t_k) \times\{1_\Gamma\}\subseteq \prod_\Gamma \cG^{(0)} \times\{1_\Gamma\}$$
and 
$$r_\Gamma(\mathcal{O}(B_1,\dots,B_k;t_1,\dots,t_k)\times\{\gamma\})=\mathcal{O}((r(B_1),\dots,r(B_k); t_1,\dots,t_k) \times\{1_\Gamma\}\subseteq\prod_\Gamma \cG^{(0)} \times\{1_\Gamma\}.$$

Note that the clopen set $\mathcal{O}(s(B_1),\dots,s(B_k);\gamma^{-1} \cdot  t_1,\dots,\gamma^{-1}\cdot t_k)$ is homeomorphic to $s(B_1)\times \cdots \times s(B_k)\times \Pi_{t\in T-\{\gamma^{-1} \cdot  t_1,\dots,\gamma^{-1}\cdot t_k\}}\cG^{(0)}$ and $\mathcal{O}((r(B_1),\dots,r(B_k); t_1,\dots,t_k)$ is homeomorphic to $r(B_1)\times \cdots \times r(B_1)\times \Pi_{t\in T-\{t_1,\dots,t_k\}}\cG^{(0)}$, which means that when we apply the product measure $\mu _T$ on the source and range, we have that 
$$\mu_T(\mathcal{O}(s(B_1),\dots,s(B_k);\gamma^{-1} \cdot  t_1,\dots,\gamma^{-1}\cdot t_k))=\mu(s(B_1))\times\cdots\times\mu(s(B_k))\times1$$
and
$$\mu_T(\mathcal{O}(r(B_1),\dots,r(B_k); t_1,\dots,t_k))=\mu(r(B_1))\times\cdots\times\mu(r(B_k))\times1.$$

So we have $\mu_T(\mathcal{O}(s(B_1),\dots,s(B_k);\gamma^{-1} \cdot  t_1,\dots,\gamma^{-1}\cdot t_k))=\mu_T(\mathcal{O}(r(B_1),\dots,r(B_k); t_1,\dots,t_k))$ which implies that $T\mathcal{G}\rtimes\Gamma$ admits an invariant measure. In this case $T\mathcal{G}\rtimes\Gamma$ can not be purely infinite.
\end{proof}

Let $\cG$ be an ample groupoid with compact unit space. Let $\Gamma$ be a group acting on a countable set $T$. We consider $M(T\cG\rtimes\Gamma)$, the set of all invariant measures of the twisted topological groupoid $T\cG\rtimes\Gamma$.

If $M(\cG)$ is non-empty, let $\mu \in M(\cG)$ be an invariant measure on $\cG^{(0)}$, then we have that the product measure $\mu_T$ is also an invariant measure on $(T\cG\rtimes\Gamma)^{(0)}$, by the proof of Theorem \ref{non-p.i.}. On the other hand, when there is only one $\cG$-invariant measure on the unit space $\cG^{(0)}$, we have that there is also a unique $(T\cG\rtimes\Gamma)$-invariant measure on $(T\cG\rtimes\Gamma)^{(0)}$.

\begin{thm}\label{unique invariant measure}
    Let $\cG$ be an ample groupoid with compact unit space. Let $\Gamma$ be a group that acts on a countable set $T$. If $|M(\cG)|=1$, then $|M(T\cG\rtimes\Gamma)|=1$.
\end{thm}
\begin{proof}
    Let $\nu$ be the unique invariant probability measure of $\cG$. It suffices to prove that for any $(T\cG\rtimes\Gamma)$-invariant measure $\nu$ on $(T\cG\rtimes\Gamma)^{(0)}$, We have that $\nu=\mu_T$. Since $\nu$ is a Borel measure, for all open sets of the form $\cO(U_1,\dots,U_k;t_1\dots,t_k)\times\{1_\Gamma\}$ in $(T\cG\rtimes\Gamma)^{(0)}$, we show that 
    \begin{align}\label{same measure}
        \nu(\cO(U_1,\dots,U_k;t_1\dots,t_k)\times\{1_\Gamma\})=\mu_T(\cO(U_1,\dots,U_k;t_1\dots,t_k)\times\{1_\Gamma\})
    \end{align}

    We prove this by induction on $k$. When $k=0$, we have that $\nu((T\cG\rtimes\Gamma)^{(0)})=\mu_T((T\cG\rtimes\Gamma)^{(0)})=1$. Suppose that (\ref{same measure}) holds for $k=n-1$. Let $\cO(U_1,\dots,U_n;t_1\dots,t_n)\times\{1_\Gamma\}$ be an arbitrary open set of $(T\cG\rtimes\Gamma)^{(0)}$ with $U_i$ open in $\cG^{(0)}$ for $i\in\{1,\dots,n\}$. Then $\cO(U_1,\dots,U_{n-1};t_1\dots,t_{n-1})\times\{1_\Gamma\}$ is an open set that satisfies the inductive assumption and we denote the $\nu$ measure of this set by $a$. We have that 
\begin{align*}
    \nu(\cO(U_1,\dots,U_{n-1};t_1\dots,t_{n-1})\times\{1_\Gamma\})
    &=\mu_T(\cO(U_1,\dots,U_{n-1};t_1\dots,t_{n-1})\times\{1_\Gamma\})\\
    &=\mu(U_1)\times\dots\times\mu(U_{n-1})\\
    &=a
\end{align*}
Now, we define a new probability measure $\mu_n$ on $\cG^{(0)}$. For any open set $W$ of $\cG^{(0)}$, define the measure of $W$ by 
\begin{align} \label{mu_n}
    \mu_n(W):=a^{-1}\cdot\nu(\cO(U_1,\dots,U_{n-1},W;t_1\dots,t_{n-1},t_n)\times\{1_\Gamma\})
\end{align}
 One can see that $\mu_n$ is a $\cG$-invariant probability measure on $\cG^{(0)}$, since every open bisection $\sigma$ on $\cG$ corresponds to an open bisection $\cO(U_1,\dots,U_{n-1},\sigma;t_1\dots,t_{n-1},t_n)\times\{1_\Gamma\}$ in $T\cG\rtimes\Gamma$. Then by the uniqueness of $\cG$-invariant measure on $\cG^{(0)}$, we have that $\mu_n=\mu$. Then taking $W=U_n$ in (\ref{mu_n}) together with (\ref{same measure}), we have that
 \begin{align*}
     \nu(\cO(U_1,\dots,U_{n-1},U_n;t_1\dots,t_{n-1},t_n)\times\{1_\Gamma\})&=a\times \mu_n(U_n)\\&=a\times\mu(U_n)\\&=\mu(U_1)\times\dots\times\mu(U_{n-1})\times\mu(U_n)\\&=\mu_T(\cO(U_1,\dots,U_{n-1},U_n;t_1\dots,t_{n-1},t_n)\times\{1_\Gamma\})  
 \end{align*}
\end{proof}
\begin{lem}\label{|M(E_n)|=1}
    Let $\cE_n$ be the approximately finite groupoid in Example \ref{ex:SFT}. Then $|M(\cE_n)|=1$.
\end{lem}
\begin{proof}
    We only prove this $\cE_2$, and for $n>2$ the proof is the same. Let $\{0,1\}^{<\infty}$ be the set of all finite dyadic words and denote $|u|$ by the word length of a word $u\in\{0,1\}^{<\infty}$. The unit space $\cE_2^{(0)}=\{0,1\}^{\bN}$ has a topological basis consisting of all cylinders $C_u$ for $u\in\{0,1\}^{<\infty}$. Here $C_u$ contains all $w\in \{0,1\}^{\bN}$ such that $u$ is a prefix of $w$.

    Let $\mu$ be an invariant measure for $\cE_2$. Since for $u,v\in\{0,1\}^{<\infty}$ with $|u|=|v|$, there exists a bisection $\sigma_{u,v}\subset\cG$ such that $s(\sigma_{u,v})=C_u$ and $r(\sigma_{u,v})=C_v$. Then for $u,v\in\{0,1\}^{<\infty}$ with $|u|=|v|$, we have $\mu(C_u)=\mu(C_v)$. From $\bigsqcup_{|u|=n}C_u=\{0,1\}^{\bN}$ and $|\{u\in\{0,1\}^{<\infty}\mid |u|=n\}|=2^n$, we have $\mu(C_u)=2^{-|u|}$ for all $u\in\{0,1\}^{<\infty}$. Hence there is only one invariant measure for $\cE_2$.
\end{proof}


\begin{thm} \label{TGxGamma is essentially principal}
     Let $\cG$ be an essentially principal minimal ample groupoid with compact unit space such that $|M(\cG)|=1$.  Let $\Gamma$ be a group that acts on an infinite countable set $T$ freely. Then the twisted topological groupoid $T\cG\rtimes\Gamma$ is essentially principal.
\end{thm}
\begin{proof}
    By Theorem \ref{unique invariant measure}, we denote by $\mu$  the unique invariant probability measure of $\cG$ and denote by $\mu_T$ the unique invariant probability measure, i.e., the product measure, of $T\cG\rtimes\Gamma$. Let $\Delta(\prod_T\cG^{(0)})$ be the diagonal of $\prod_T\cG^{(0)}$, defined by 
    $$\Delta(\prod_T\cG^{(0)}):=\{\kappa\in\prod_T\cG^{(0)}\mid \exists\ t\neq t^\prime\in T, \kappa(t)=\kappa(t^\prime)\}$$

    For $t\neq t'\in T$, denote $\Delta(t,t^\prime):=\{\kappa\in\prod_T\cG^{(0)}\mid \kappa(t)=\kappa(t^\prime)\}$, Therefore, one has
    \[\Delta(\prod_T\cG^{(0)})=\bigcup_{t\neq t'\in T}\Delta(t,t^\prime).\] Since $T$ is a countable set, there are countably many such $\Delta(t,t^\prime)$. And for each $\Delta(t,t^\prime)$, we have \[\Delta(t,t^\prime)=\{(x,y)\in\cG^{(0)}\times\cG^{(0)}\mid x=y\}\times\prod_{T-\{t,t^\prime\}}\cG^{(0)}\] and thus
    \[\mu_T(\Delta(t,t^\prime))=(\mu\times \mu)(\{(x,y)\in\cG^{(0)}\times\cG^{(0)}\mid x=y\}).\] Using the minimality of $\cG$, we have $\mu(\{x\})=0$ for all $x \in\cG^{(0)}$. Hence, by Tonelli's theorem, we have $(\mu\times \mu)(\{(x,y)\in\cG^{(0)}\times\cG^{(0)}\mid x=y\})=0$, which entails $\mu_T(\Delta(\prod_T\cG^{(0)}))=0$.

Now for $\kappa\in\prod_T\cG^{(0)}-\Delta(\prod_T\cG^{(0)})$, we consider the isotropy group of $(\kappa,1_\Gamma)\in (T\cG\rtimes\Gamma)^{(0)}$. For any $(f,\gamma)\in T\cG\rtimes\Gamma$ with $s_\Gamma(f,\gamma)=(\kappa,1_\Gamma)=r_\Gamma(f,\gamma)$, we claim that $\gamma=1_\Gamma$. Since for the map $f$, there exists $t_0\in T$ such that $f(t_0)\in\cG^{(0)}$. And applying the source and range maps to $(f,\gamma)$, we have 
\begin{align*}
    s_\Gamma(f,\gamma)=(\gamma^{-1}\cdot s(f),1_\Gamma)=r_\Gamma(f,\gamma)=(r(f),1_\Gamma)=(\kappa,1_\Gamma)
\end{align*}
Then we have that 
\begin{align}\label{kappa}
\kappa(t_0)=s(f(\gamma\cdot t_0))=r(f(t_0))=f(t_0)=s(f(t_0))=\kappa(\gamma^{-1}\cdot t_0)
\end{align}
 If $\gamma\neq 1_\Gamma$, the free $\Gamma$-action implies that $\gamma\cdot t_0\neq t_0$, which contradicts to $\kappa\in\prod_T\cG^{(0)}-\Delta(\prod_T\cG^{(0)})$ by (\ref{kappa}). 

 So the elements in the isotropy group of $(\kappa,1_\Gamma)\in (T\cG\rtimes\Gamma)^{(0)}$ have the form $(f,1_\Gamma)\in T\cG\rtimes\Gamma$. Therefore, the element $(\kappa,1_\Gamma)\in (T\cG\rtimes\Gamma)^{(0)}$ has non-trivial isotropy group if and only if there exists some $t^\prime\in T$ such that $\kappa(t^\prime)$ has non-trivial isotropy group in $\cG$. Denote by $R$ the subset of $\cG^{(0)}$ with non-trivial isotropy group. Since $\cG$ is essentially principal, one has $\mu(R)=0$ and thus $\mu_T(\cO(R;t^\prime))=\mu(R)=0$. This implies that the set 
 \[\{x\in(T\cG\rtimes\Gamma)^{(0)}: \cG^x_x\neq \{x\}\}\subset\bigcup_{t\in T}\cO(R;t)\cup \Delta(\prod_T\cG^{(0)})\]
and thus has measure $0$. 
\end{proof}

\begin{proof}[Proof of Theorem \ref{main:thmb}]
    By Proposition \ref{from G to TGxGamma}, $\Gamma\cG\rtimes \Gamma$ is a minimal ample groupoid. From $|M(\cG)|=1$ and every single point has measure zero, by Theorem \ref{TGxGamma is essentially principal}, $\Gamma\cG\rtimes \Gamma$ is essentially principal. Since $\Gamma$ is not amenable, by Theorem \ref{non alomst finiteness}, $\Gamma\cG\rtimes \Gamma$ is not almost finite. The fact that $\Gamma\cG\rtimes \Gamma$ is not purely infinite follows from Theorem \ref{non-p.i.}.
\end{proof}

\begin{rem}
   The proof also implies that if  $\Gamma$ is a countable non-amenable group and if $\cG$ is an effective and minimal almost finite groupoid, then the twisted topological groupoid $\Gamma\cG\rtimes \Gamma$ is an effective minimal ample groupoid that is neither almost finite nor purely infinite.
\end{rem}

\begin{ex} \label{ex:non a.f. and non p.i.}
    Let $\Gamma$ be a non-amenable group  and $T=\Gamma$ with $\Gamma$ acting on it by left multiplication. Let $\cE$ be a minimal approximate finite groupoid on a compact unit space with a unique $\cE$-invariant probability measure, e.g., UHF groupoids  $\cE_n$ in Example \ref{|M(E_n)|=1} or certain non-UHF groupoids $\cE$ yielding simple unital AF $C^*$-algebra $C^*_r(\cE)$ as in \cite[Corollary 3.7]{Bl80}.
  Note that $\cE$ is principal as well.    
    Thus the twisted topological groupoid $\Gamma\cE\rtimes \Gamma$ is an essentially principal, minimal, ample groupoid that is neither almost finite nor purely infinite. 
\end{ex}

\begin{thm}\label{thm: non trans}
    Suppose $\cG$ is a principal almost finite non-topologically amenable groupoid. Let $\Gamma$ be a countable discrete group. Then $\Gamma \cG$ is not isomorphic to any transformation groupoid. And if $\Gamma$ is not amenable, $\Gamma\cG\rtimes\Gamma$ is an effective groupoid that is not almost finite nor purely infinite.
\end{thm}
\begin{proof}
    Let $\cG$ be an almost finite groupoid that is not topologically amenable. Then so is $\Gamma\cG$ by Proposition \ref{from G to TG} (6). Then $\Gamma\cG$ is not isomorphic to any transformation groupoid. Otherwise, suppose $\Gamma\cG\simeq X\rtimes \Lambda$. Then almost finiteness of $\Gamma\cG$ implies that $\Lambda$ is amenable  and therefore $\Gamma\cG\simeq X\rtimes \Lambda$ has to be topologically amenable. But this is a contradiction. Furthermore, since $\cG$ is principal, Proposition \ref{from G to TG} shows that $\Gamma\cG$ is principal as well. Finally, it follows from Theorems \ref{non alomst finiteness},  \ref{non-p.i.} and Proposition \ref{from G to TGxGamma} that $\Gamma\cG\rtimes \Gamma$ is effective and is neither almost finite nor purely infinite.
\end{proof}

\begin{ex}\label{ex: TG non transformation}
    Let $\cG_0$ be the minimal principal second countable ample almost finite \'{e}tale \textit{geometric groupoid} constructed in \cite{El18} that is not topologically amenable and $\Gamma$ a non-amenable group.  Then by Theorem \ref{thm: non trans}, the groupoid $\Gamma\cG_0$ is not isomorphic to any transformation groupoid and  $\Gamma\cG_0\rtimes \Gamma$ is a minimal effective second countable ample \'{e}tale groupoid that is neither almost finite nor purely infinite.
\end{ex}

We end the paper with the following remark.

\begin{rem}\label{rmk: final remark}
    We note that all groupoids constructed in Theorem \ref{main:thmb}—and therefore, in particular, those in Examples \ref{ex:non a.f. and non p.i.} and \ref{ex: TG non transformation}—are not transformation groupoids \textit{themselves}, as they are of the form $\Gamma\cG\rtimes \Gamma$ with $\Gamma\cG\neq(\Gamma\cG)^{(0)}$. However, some of them are isomorphic to transformation groupoids. For instance, let $\Gamma$ be a non-amenable group and let $\cE$ be a UHF groupoid such that $C^*_r(\cE)\simeq \bigotimes_{n=1}^\infty M_{k_n}(\mathbb{C})$ (for example, $C^*_r(\cE_2)\simeq M_{2^\infty}(\mathbb{C})$). Then $\cE$ is isomorphic to the transformation groupoid $X\rtimes \Lambda$ arising from the natural action of $\Lambda=\bigoplus_{n=1}^{k_n}\bZ/k_n\bZ$ on the Cantor set $X=\prod_{n=1}^\infty \bZ/k_n\bZ$ by \cite{Kri80} (see also \cite[Theorem 4.10]{Matui2012}). It follows that $\Gamma\cE\rtimes\Gamma \simeq X^\Gamma\rtimes (\Lambda\wr_\Gamma \Gamma)$.

It remains unknown to the authors, however, whether $\Gamma\cG\rtimes \Gamma$ fail to be isomorphic to a transformation groupoid when $\cG$ itself fails. This situation may occur even when $\cE$ is an approximately finite groupoid that is not UHF, such as certain approximately finite groupoids from \cite[Corollary 3.7]{Bl80}, mentioned in Example \ref{ex:non a.f. and non p.i.} above.  See also \cite[Theorem 4.2,4.3]{Dahl08} for a characterization of when minimal approximately finite topological groupoids are isomorphic to transformation groupoids.  On the other hand, in Example \ref{ex: TG non transformation}, we already know that $\Gamma\cG_0$ is not isomorphic to any transformation groupoid. Nevertheless, whether the same holds for $\Gamma\cG_0\rtimes \Gamma$ remains unresolved.
    
In fact, determining whether a given (effective) groupoid is isomorphic to a transformation groupoid is generally a very difficult problem. From the $C^*$-algebraic perspective, this question is also related to whether the Jiang--Su algebra $\mathcal{Z}$ can be realized as a crossed product $C^*$-algebra---a well-known open problem in the structure theory of $C^*$-algebras. We refer the reader to \cite{MW24} for recent progress on this question.
\end{rem}

\bibliographystyle{alpha}
\bibliography{references.bib}

\end{document}